\documentclass[12pt]{amsart}
\usepackage{mathptmx}
\usepackage[curve,tips]{xypic}
\SelectTips{eu}{11}
\UseTips
\usepackage{amsmath,amsthm,amssymb}
\usepackage{pdfsync}
\usepackage[usenames,dvipsnames]{color}

\usepackage{manfnt}
\usepackage{stmaryrd}

\newtheorem{theorem}[subsection]{Theorem}
\newtheorem{proposition}[subsection]{Proposition}
\newtheorem{corollary}[subsection]{Corollary}
\newtheorem{lemma}[subsection]{Lemma}
\newtheorem{conjecture}[subsection]{Conjecture}

\newtheorem{addendum}[subsection]{Addendum}

\theoremstyle{definition}

\newtheorem{definition}[subsection]{Definition}

\newtheorem{question}[subsection]{Question}

\theoremstyle{remark}
\newtheorem{remark}[subsection]{Remark}
\newtheorem{remarks}[subsection]{Remarks}

\newcommand{\normal}{\triangleleft}
\newcommand{\mono}{\rightarrowtail}
\newcommand{\epi}{\twoheadrightarrow}
\newcommand{\iso}{\cong}

\newcommand{\Q}{{\mathbb Q}}
\newcommand{\Z}{{\mathbb Z}}

\newcommand{\N}{{\mathbb N}}
\newcommand{\F}{{\mathbb F}}

\newcommand{\FP}{\operatorname{FP}}

\newcommand{\cd}{\operatorname{cd}}
\newcommand{\hd}{\operatorname{hd}}

\newcommand{\fdhd}{\operatorname{fdhd}}
\newcommand{\cdq}{\operatorname{cd}_\Q}
\newcommand{\hdq}{\operatorname{hd}_\Q}
\newcommand{\ext}{\operatorname{Ext}}

\renewcommand{\>}{\rangle}
\newcommand{\blah}{\quad}

\newcommand{\colim}{\varinjlim{}}
\renewcommand{\lim}{\varprojlim{}}

\newcommand{\cll}{{\scriptstyle\bf L}}

\title{Dimension of elementary amenable groups}
\author{M. R. Bridson}
\author{P. H. Kropholler}
\address{\newline Mathematical Institute, 
\newline 24--29 St Giles', 
\newline Oxford OX1 3LB}
\email{ bridson@maths.ox.ac.uk}
\address{\newline School of Mathematics and Statistics, 
\newline University of Glasgow, 
\newline Glasgow G12 8QW}
\email{ peter.kropholler@glasgow.ac.uk}
\thanks{
The first author was supported in part by an EPSRC Senior Fellowship 
and a Royal Society Wolfson Research Merit Award. 
The second author was supported in part by a Visiting Professorship at Cornell University.
Both authors thank the Mittag Leffler Institute, Djursholm, for its
hospitality during the preparation of this manuscript.
}


\begin{document}

\begin{abstract} This paper has three parts.  
It is conjectured that for every elementary amenable group $G$ and every non-zero commutative ring $k$, 
the homological dimension $\hd_{k}(G)$ is equal to the Hirsch length $h(G)$ whenever $G$ has no $k$-torsion. 
In Part I this conjecture is proved for several classes, including the abelian-by-polycyclic groups. In Part II it is shown that the elementary amenable groups of homological dimension one are colimits of systems of groups of cohomological dimension one. In Part III the deep problem of 
calculating the
cohomological dimension of elementary amenable groups is tackled with particular emphasis on the nilpotent-by-polycyclic case, where
a complete answer is obtained over $\Q$ for countable groups.
\end{abstract}

\maketitle

\section*{Introduction} 
This paper is divided into three parts, each devoted to an aspect of the theory of (co)homological dimension for discrete soluble groups.  
Wherever possible,  we have aimed for results that hold in the class of elementary amenable groups, since this class may be regarded as the largest class of groups that warrant the title of generalized soluble groups; by definition it is 
the smallest class of groups that contains all finite and abelian groups and is closed under extensions and 
directed unions.  
Homological and cohomological dimension, though, are subtle invariants 
that oblige us to retain additional hypotheses at various stages of our treatment.

Throughout, $k$ will denote a non-zero commutative ring. The \emph{cohomological dimension} $\cd_k(G)$ over $k$ of a group $G$ is the least integer $m$ such that $H^{m+1}(G,M)=0$ for all $kG$-modules $M$; or $\infty$ if no such $m$ exists. The homological dimension is defined similarly, using homology functors in place of cohomology functors. 
It is always the case that $\hd_kG\le\cd_kG$. If $G$ is countable then $\cd_{k}(G)\le\hd_{k}(G)+1$.
If $k'$ is a second non-zero commutative ring and $k\to k'$ is a ring homomorphism then $\hd_{k'}\le\hd_{k}$ and $\cd_{k'}\le\cd_{k}$;
so in particular $\hd_k \le \hd_\Z$ and $\cd_k\le \cd_\Z$. 
We say that a group has no $k$-torsion when the orders of its finite subgroups are invertible in $k$. 
For any group $G$, the finiteness of $\hd_{k}(G)$ implies that $G$ has no $k$-torsion. We refer the reader to Bieri's notes \cite{bieriqmc} for proofs
of these basic facts and other background material.

In Part I of this paper we study the homological dimension of elementary amenable groups over arbitrary coefficient rings. 
We have found that even in the classical soluble case there is a need for clarification of the existing literature. Our main result in
this direction is Theorem \ref{hd}, which shows that for abelian-by-polycyclic groups $G$, the homological dimension is equal to the Hirsch length $h(G)$ when both quantities are finite. In particular, when finite, the homological dimension of such a group is independent of $k$. 

We conjecture that the homological dimension and Hirsch length always coincide for elementary amenable groups whenever the former is finite. This conjecture is claimed for arbitrary soluble groups by Fel{$'$}dman \cite{feldman1971} but the argument there is flawed and it seems unlikely, as we shall see, that there is any elementary way of repairing it.
We have already noted that for any group $G$, the finiteness of $\hd_{k}(G)$ implies that $G$ has no $k$-torsion. For soluble groups without $k$-torsion, one also has the  well-known inequality 
$\hd_{k}(G)\le h(G)$ (see \cite{bieriqmc}, Proposition 7.11, for example), 
and this remains valid for elementary amenable groups. Moreover,
the equality $\hd_{k}(G)=h(G)$  has been firmly 
established \emph{in the case when $k$ is a field of characteristic zero} (\cite{stammbach}).

The second part of the paper begins 
 with a series of examples that motivate the use of filtered colimit systems. We go on to consider the conjecture that 
arbitrary 
groups of cohomological dimension one over $\Z$ are locally free, as well as 
natural generalizations for other coefficient rings. We shall prove that this conjecture holds for elementary amenable groups. 
We also explain why, for integers $\ell >1$, one cannot characterise groups of cohomological dimension $\ell$ in terms of directed colimits.  

In the third part we attack the much tougher problem of characterizing
cohomological dimension 
for elementary amenable groups over arbitrary coefficient rings. We present a natural conjecture and establish some technical reductions of it.
The nilpotent-by-polycyclic case is investigated in greater detail with emphasis on the 
 coefficient ring $k=\Q$. In this generality, we establish the main conjecture for countable groups and
present a significant reduction in the case of groups of greater cardinality.  

We thank the referee for reading our original draft so carefully and for providing such constructive comments.

\section{Homological Dimension}
We begin with 
 an account of the following conjecture, discussing previous work in this direction and presenting partial results concerning group algebras of soluble groups.  
Later we shall enlarge the discussion to cover elementary amenable groups. We assume that the reader is familiar with the notion of
Hirsch length in the setting of solvable groups; it is discussed in greater generality in (\ref{def:hirschlength}). 

\begin{conjecture}\label{hdconj}
Let $G$ be an elementary amenable group with no $k$-torsion. Then $\hd_{k}(G)=h(G)$.
\end{conjecture}

We wish to emphasize that this conjecture is open for soluble groups. The statement for elementary amenable groups can be shown to follow easily from the soluble case. The condition of no $k$-torsion was discussed in 
the introduction. Beyond this, the first thing that we shall prove is the following:

\begin{theorem}\label{hd0}
Let $G$ be an elementary amenable group with no $k$-torsion. 
Then $h(G)/2\le \hd_{k}(G)\le h(G)$
\end{theorem}

In particular, this shows that the finiteness of $\hd_{k}(G)$ is equivalent to finiteness of $h(G)$ and so for torsion-free elementary amenable groups, finiteness of homological dimension over $k$ is independent of $k$. 

Of course we expect the equality $\hd_{k}(G)=h(G)$ whenever $\hd_{k}(G)$ is finite. We establish this in a special case by using some of the techniques introduced by Fel{$'$}dman.

\begin{theorem}\label{hd}
Conjecture \ref{hdconj} holds for abelian-by-polycyclic-by-finite groups.
\end{theorem}

\begin{remark}
\label{simp} For any group $G$, $\hd_{k}(G)$ is equal to the supremum of $\hd_{k}(H)$ as $H$ runs through the finitely generated subgroups of $G$. Naturally this is useful in proving Theorem \ref{hd}.
This is false for cohomological dimension, a point which we discuss more fully in Part II.

It is widely recognized that the homological (or weak) dimension is more easily understood than the cohomological dimension in the context of group rings. 
\end{remark}

Conjecture \ref{hdconj} is similar in form to many well established results. Stammbach proved this for soluble groups in case $k$ is a field of characteristic zero, (\cite{stammbach}, Theorem 1). A very readable account may be found in Bieri's notes (\cite{bieriqmc}, \S7.3). The special case of Theorem \ref{hd} in which $k=\Z$ appears in Bieri's notes (\cite{bieriqmc}, Theorem 7.10a). Bieri's account follows similar lines to Stammbach's original work, using along the way the fact that if $G$ is a torsion-free nilpotent group of Hirsch length $n$ then the homology groups $H_{n}(G,\Z)$ and $H_{n}(G,\Q)$ are both  non-zero. Gruenberg also gives a useful account \cite{gruenberg} and goes further by determining both the cohomological dimension and 
 the homological dimension for nilpotent groups.

Stammbach's argument leads to a weaker conclusion over fields of positive characteristic. The reason is that $p$-divisibility in the group results in premature vanishing of homology with trivial coefficients of characteristic $p$. For example, if $G$ is the additive group $\Z[1/p]$ and $\F_{p}$ is the field of $p$ elements then $H_{1}(G,\F_{p})=0$ while it is obvious that $\hd_{\F_{p}}(G)=1$ because $G$ has an infinite cyclic subgroup which already has homological dimension one over any coefficient ring. The issue becomes more acute when one embeds $\Z[1/p]$ into a finitely generated group. The one-relator group $H:=\langle x,y;\ x^{-1}yx=x^{p}\rangle$ contains a copy of $\Z[1/p]$ generated by the conjugates of $y$. Moreover $H$ is a metabelian group and so falls within the remit of Theorem \ref{hd}. It is indeed true that $\hd_{k}(H)$ equals the Hirsch length $h(H)=2$ no matter what the choice of $k$, but proving this appears to require ideas beyond those presented by Bieri, Gruenberg or Stammbach for dealing with $k=\Z$ or $k=\Q$.

On the other hand, there is another paper \cite{feldman1971} by Fel{$'$}dman which builds on Stammbach's work \cite{stammbach} and cites Gruenberg's notes \cite{gruenberg}. 
 Fel{$'$}dman announces Theorem \ref{hd} above but the proof is flawed, as we shall explain. 
  At the same time, Fel{$'$}dman's paper contains some very important insights amongst the most crucial of which lead to the following powerful equations.

\begin{theorem}[Fel{$'$}dman (\cite{feldman1971}, Theorem 2.4), Bieri (\cite{bieriqmc}, Theorem 5.5)]\label{feld} 
Let $N\mono G\epi Q$ be a short exact sequence of groups in which $N$ is of type $\FP$ over $k$. Suppose that $H^{n}(N,kN)$ is free as a $k$-module for $n=\cd_{k}(N)\ (=\hd_{k}(N))$. Then
\begin{enumerate}
\item
if $\cd_{k}(Q)<\infty$ then $\cd_{k}(G)=\cd_{k}(N)+\cd_{k}(Q)$, and
\item
if $\hd_{k}(Q)<\infty$ then $\hd_{k}(G)=\hd_{k}(N)+\hd_{k}(Q)$.
\end{enumerate}
\end{theorem}

Fel{$'$}dman states this result in case $k$ is a field and Bieri establishes the generalization to arbitrary $k$. In both cases of course the inequalities $\le$ follow at once from standard spectral sequence arguments.

Fel{$'$}dman goes further and states that if $G$ is soluble without $k$-torsion then $\hd_{k}(G)=h(G)$, (see \cite{feldman1971}, Theorem 3.6). However the proof offered includes the sentence ``Stammbach's proof \cite{stammbach} $\ldots$ applies for any field $\dots$'', which is an assertion we have seen to fail in some cases. For this reason we revisit Fel{$'$}dman's claim in the context of Theorem \ref{hd} and provide a new proof that uses the Fel'dman--Bieri Theorem \ref{feld} but avoids appealing to any supposed variation of Stammbach's characteristic zero argument.

It seems very unlikely that Stammbach's proof could ever have a generalization in the direction suggested by Fel{$'$}dman
because Stammbach's argument uses finite dimensional modules. Let us examine this point more closely. In case $k$ is a field, let $\fdhd_{k}(G)$ denote $\sup\{n\ge0;\ H_{n}(G,M)\ne0\textrm{ for some $kG$-module $M$ satisfying }\dim_{k}M<\infty\}$, the \emph{$k$-finite-dimensional homological dimension}.  Stammbach's argument rests on showing that if $G$ is a soluble group of finite Hirsch length $h$ then $\fdhd_{\Q}(G)\ge h$; 
he does this by explicitly constructing a module $M$ so that the $h$th homology is non-zero. 
By combining this information with the more elementary inequality $\hd_{\Q}(G)\le h$ and the trivial observation $\fdhd_{\Q}\le\hd_{\Q}$, one  concludes that $\hd_{\Q}(G)=h$. But, crucially, 
 if we consider finite fields $k$ and soluble minimax groups $G$, the $k$-finite-dimensional homological dimension cannot be a strong enough invariant for the task in hand, as we shall now explain. 
 We shall need the following easy lemma in our discussion.

\begin{lemma} \label{fdfi}
Let $k$ be a field.
Let $H$ be a subgroup of finite index in a group $G$ with $\fdhd_{k}(G)<\infty$. Then $\fdhd_{k}(H)=\fdhd_{k}(G)$.
\end{lemma}
\begin{proof}
The inequality $\fdhd_{k}(H)\le\fdhd_{k}(G)$ follows from the Eckmann--Shapiro  lemma
because the  induction functor takes finite dimensional $kH$-modules to finite dimensional $kG$-modules. If $n=\fdhd_{k}(G)$ is finite then choose a finite dimensional $kG$-module $M$ such that $H_{n}(G,M)$ is non-zero. Now there is a natural embedding of $M$ into the induced module $M\otimes_{kH}kG$ and so the tail end of the long exact sequence of homology for $G$ shows that $H_{n}(G,M\otimes_{kH}kG)$ is non-zero. Thus, again by the Eckmann--Shapiro lemma, we have $H_{n}(H,M)$ is non-zero and $\fdhd_{k}(H)\ge n$.
\end{proof}

Note that the inequality $\fdhd_{k}(H)\le\fdhd_{k}(G)$ from the beginning of this proof really does depend on finite index: we shall see this point in more detail shortly.

\subsection{A review of the theory of soluble minimax groups} A \emph{soluble minimax group} is a group with a series $1=G_{0}\normal G_{1}\normal\dots\normal G_{r}=G$ of finite length in which the factors $G_{i}/G_{i-1}$ are cyclic or quasicyclic. 
An elementary amenable minimax group allows arbitrary finite factors in addition. Both classes are subgroup, quotient and extension closed.
Every elementary amenable minimax group is virtually soluble, in fact virtually nilpotent-by-abelian. 
Every virtually torsion-free elementary amenable minimax group 
is nilpotent-by-(finitely generated abelian)-by-finite. If $p$ and $q$ are distinct primes then there is a faithful action of $\Z[1/q]$ on $C_{p^{\infty}}$ and the resulting semidirect product 
$C_{p^{\infty}}\rtimes\Z[1/q]$ is an example to illustrate that in general, while always virtually nilpotent-by-abelian, soluble minimax groups need not be virtually nilpotent-by-(finitely generated abelian). Accounts of the theory may be found in \cite{djsr1,djsr2,lennoxrobinsonbook}.

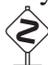
\textdbend
Diverting from Robinson's terminology in a way that we hope is harmless in this context, we shall use the term \emph{minimax} always to refer to elementary amenable minimax groups. (Robinson allows for other, more exotic groups 
 to be included but this has no relevance here.)

The {\em{Hirsch length}} 
 of a minimax group is simply the count of the number of infinite cyclic factors in a cyclic/quasicyclic/finite series
(cf. \ref{def:hirschlength}). 
 We can also define an invariant $h^{*}$ for minimax groups by counting the number of infinite factors, and for each prime $p$ we can count the number $m_{p}$ of factors isomorphic to $C_{p^{\infty}}$. It is sometimes useful to consider the set $\pi$ of primes $p$ for which $G$ has a section isomorphic to $C_{p^{\infty}}$. The independence of these counts from choice of series follows at once from the Schreier refinement theorem. Note that we also have $h^{*}=h+\sum_{p}m_{p}$ and $\pi=\{p;\ m_{p}\ne0\}$.

\begin{proposition}\label{propI.8}
Let $p$ be a prime and let $\F_{p}$ denote the finite 
 field of $p$ elements.
Let $G$ be a torsion-free soluble minimax group. Then $\fdhd_{\F_{p}}(G)=h(G)-m_{p}(G)$.  
\end{proposition}
\begin{proof}
We may choose a series $1=G_{0}\le G_{1}\le\dots\le G_{r}$ terminating in some subgroup $G_{r}$ of finite index in $G$ in which the factors are torsion-free abelian of rank $1$. Lemma \ref{fdfi} shows that, without loss of generality, we may replace $G$ by $G_{r}$ and so assume that $G=G_{r}$. We establish the result by induction on $r$. This follows easily from the following facts.
\begin{enumerate}
\item
If $M$ is a finite $\F_{p}G$-module 
 then $H_{j}(G,M)$ is finite for all $j$.
\item
If $G$ is torsion-free abelian of rank one and $p$-divisible (or equivalently has a section isomorphic to $C_{p^{\infty}}$) then $H_{j}(G,M)=0$ for all finite $\F_{p}G$-modules $M$ and all $j>0$.
\end{enumerate}
(i) is easily proved using a spectral sequence argument. (ii) follows from straightforward calculation. First, the use of long exact cohomology sequences can be used to reduce to the case when $M$ is an irreducible $G$-module. Since $G$ is abelian, all the cohomology groups with coefficients in a non-trivial irreducible $\F_{p}G$-module are zero. Thus the only case where calculation is required is when $M=\F_{p}$ is the trivial module and in this case a universal coefficient theorem together with the $p$-divisibility of $G$ shows that cohomology vanishes in dimensions greater than zero.
\end{proof}

Combining Proposition \ref{propI.8} with the trivial observation that
the inequality $\fdhd_{k}(G)\le\hd_{k}(G)$ holds for any $k$, we see that in positive characteristic Stammbach's argument leads naturally to the following conclusion.

\begin{lemma}\label{LI7}
If $G$ is a torsion-free soluble minimax group then
for each prime $p$, $\hd_{\F_{p}}(G)\ge h(G)-m_{p}(G)$.
\end{lemma}


We have been unable to see how Stammbach's arguments could be adapted to yield the conjectured equality 
$\hd_{\F_{p}}(G)=h(G)$.


This concludes our discussion of Fel'dman's approach to Conjecture \ref{hdconj}. Henceforth we concentrate on improving 
the inequality in Lemma \ref{LI7} under additional hypotheses.

\subsection{Constructible soluble groups and inverse duality}\ 

Constructible groups were introduced by Baumslag and Bieri \cite{BaumslagBieri1976}. 
Constructible amenable groups are elementary amenable and virtually of type ${\rm{F}}$. 
They
can be built up from the trivial group by a finite series of ascending HNN-extensions and finite extensions.  It is now known that all elementary amenable groups of type $\FP$ are constructible, and we refer the reader to \cite{krophollermartinezpereznucinkis} for an up to date account of many homological properties and characterizations of these groups. 

\begin{proposition}\label{constr}
If $G$ is a torsion-free constructible soluble group then $\hd_{k}(G)=\cd_{k}(G)=h(G)$.
\end{proposition}
\begin{proof} It follows from results of Gildenhuys and Strebel \cite{gildenhuysstrebel1982} that
a torsion-free constructible soluble group $G$ satisfies $\hd_{\Z}(G)=\cd_{\Z}(G)=h(G)<\infty$. According to Theorem 9 of \cite{BaumslagBieri1976} torsion-free constructible soluble groups are inverse duality groups over $\Z$ in the sense of \S9.3 of \cite{bieriqmc}, that is to say these groups are duality groups over $\Z$ whose dualizing modules are free abelian. Therefore, writing $n=h(G)$, 
we 
have that $H^n(G, kG) \cong H^n(G,\Z G)\otimes k\neq 0$ for any non-zero commutative ring $k$. Thus $\cd_{k}(G)=\cd_{\Z}(G)$,
and by duality $\hd_{k}(G)=\hd_{\Z}(G)$.
\end{proof}

This enables us to settle Conjecture \ref{hdconj} for constructible virtually soluble groups.

\begin{corollary}\label{constrcor}
If $G$ is constructible and virtually soluble then $\hd_{k}(G)$ is finite if and only if $G$ has no $k$-torsion and has finite Hirsch length. Moreover, when these conditions hold,  $\hd_{k}(G)=h(G)$.
\end{corollary}
\begin{proof}
Constructible virtually soluble groups contain torsion-free subgroups of finite index to which the Proposition can be applied. The finite index step up is standard, resting on  
 two observations about a subgroup $H$ of finite index in a group $G$:
\begin{enumerate}
\item
if $\hd_{k}(G)<\infty$ then $\hd_{k}(G)=\hd_{k}(H)$ --- a simple consequence of the Shapiro--Eckmann lemma; and
\item
if $G$ has no $k$-torsion and $\hd_{k}(H)<\infty$ then $\hd_{k}(G)<\infty$ --- this follows from a variation on the K\"unneth theorem for tensor powers of resolutions.
\end{enumerate}
\end{proof}

\subsection{Elementary amenable groups}\label{eagroups} 

Having seen that Conjecture \ref{hdconj} 
holds for constructible virtually soluble groups we wish to go further and consider the situation for the much wider class of elementary amenable groups. 
The reader will recall that the class of elementary amenable groups is defined to be the smallest class of groups which contains all finite and all abelian groups and which is closed under extensions and directed unions. All groups in this class are amenable.  They have received attention in part because they can be studied using traditional methods from the theory of generalized soluble groups, but also because at the same time they form a fascinatingly complicated class
of groups encompassing a great range of behaviour. 
We have the following version of Conjecture \ref{hdconj} for these groups.

\begin{conjecture}\label{mainI}
Let $G$ be an elementary amenable group and let $k$ be a non-zero commutative ring. Then $\hd_{k}(G)$ is finite if and only if $G$ has no $k$-torsion and finite Hirsch length. When these conditions hold we have $\hd_{k}(G)=h(G)$.
\end{conjecture}

The definition of Hirsch length is well known for soluble groups but may need some elaboration for elementary amenable groups in general.
We formulate it as follows.

\begin{definition}\label{def:hirschlength}
The  
 \emph{Hirsch length $h(G)$} of an elementary amenable group is finite and equal to $n\ge0$ if $G$ has a series $1=G_{0}\normal G_{1}\normal\cdots\normal G_{r}=G$ in which the factors are either locally finite or infinite cyclic, and exactly $n$ factors are infinite cyclic. In all other cases, $h(G)=\infty$. Note that the Schreier refinement theorem shows this to be an invariant of the group independent of any choice of subnormal series.
\end{definition}

That this definition coincides with that of Hillman and Linnell \cite{hillmanlinnell} follows from Lemma \ref{thelemma:hirschlength} below. To explain this we appeal to an argument of Wehrfritz:

\begin{proposition}[Wehrfritz \cite{Wehrfritz1995}, (g)]
\label{bafw95g}
There is an integer-valued function $j(h)$ of $h$ only such that an elementary amenable group $G$ with $h(G)<\infty$ has characteristic subgroups $\tau(G)\le N \le M$ with $\tau(G)$ torsion, $N/\tau(G)$ torsion-free nilpotent, $M/N$ free abelian of finite rank and $|G : M|$ at most $j(h(G))$.
\end{proposition}

\begin{lemma}\label{thelemma:hirschlength}
Let $G$ be an elementary amenable group of infinite Hirsch length. Then there is a chain of subgroups $G_{0}<G_{1}<\dots<G_{n}<\cdots$ (indexed by the natural numbers) for which the sequence of Hirsch lengths $(h(G_{n}))$ increases without bound.
\end{lemma}
\begin{proof}
For each ordinal $\alpha$, let $\mathcal X_{\alpha}$ be the class of elementary amenable groups defined by Hillman and Linnell \cite{hillmanlinnell}. Recall that $\mathcal X_{0}$ is the trivial groups' class, $\mathcal X_{1}$ consists of all abelian-by-finite groups, and $\mathcal X_{\alpha}$ is defined to be $(\cll\mathcal X_{\alpha-1})\mathcal X_{1}$ if $\alpha$ is a successor and $\bigcup_{\beta<\alpha}\mathcal X_{\beta}$ if $\alpha$ is a limit ordinal. 
(Here, $(\cll\mathcal X_{\alpha-1})\mathcal X_{1}$ is the class of those groups that are an extension of a group which is locally $\mathcal X_{\alpha-1}$ by a group which is in $\mathcal X_{1}$.)
A group is elementary amenable if and only if it belongs to some $\mathcal X_{\alpha}$, and the least such $\alpha$ provides a \emph{level} allowing arguments by transfinite induction. The level of $G$ is the least $\alpha$ for which $G$ belongs to $\mathcal X_{\alpha}$;
this is always a successor ordinal when $G$ is non-trivial.

Towards 
 a contradiction, let $G$ be an elementary amenable group of infinite Hirsch length which does not possess a strictly ascending chain of subgroups as is required by the lemma. Let $\alpha$ be the level of $G$. We may assume that groups of level less than $\alpha$ either have finite Hirsch length or admit an instance of the desired chain. If the latter happens then we are done. Hence we may assume that every subgroup of $G$ of level less than $\alpha$ has finite Hirsch length. If there is a bound on the Hirsch numbers of those subgroups of $G$ that have finite Hirsch length, then it follows that $G$ itself has finite Hirsch length (see Wehrfritz' argument, Proposition \ref{bafw95g}). So we may assume that there is no such bound on Hirsch lengths and now it is clear that we can find a chain of the required type.
\end{proof}

\subsection{A lower bound for homological dimension}

\begin{proof}[Proof of Theorem \ref{hd0}]
First note that, without interfering with the hypotheses,
 we may replace $G$ by $G/T$ where $T<G$ is the largest normal locally finite subgroup. So we assume that $T=1$. 
By Remark \ref{simp} 
there is no loss of generality in assuming that $G$ is finitely generated and therefore 
virtually soluble and minimax.
According to Lennox and Robinson (\cite{lennoxrobinson}, Theorem C), there are nilpotent subgroups $N$ and $S$ with $N$ normal and $S$ finitely generated such that $NS$ has finite index in $G$. Since $N$ is locally polycyclic it is clear that $\hd_{k}(N)= h(N)$. We also have 
$\hd_{k}(S)=h(S)$ (by \ref{constrcor}, for example). 
 The result follows because $\hd_{k}(G)$ is bounded below by both $\hd_{k}(N)$ and $\hd_{k}(S)$ and at least one of these is $\ge h(G)/2$.
\end{proof}

\subsection{Homological dimension in the abelian-by-polycyclic case}\label{LR}\ 

The Lennox--Robinson result can also be used to make another reduction. 
More generally, suppose $G=NS$ is the product of a normal nilpotent minimax group $N$  and a polycyclic group $S$.
Then $S$ acts on $N$ by conjugation and we may form the semidirect product $N\rtimes S$ which then admits a homomorphism to $G$ given by $(s,n)\mapsto sn$. The kernel of this homomorphism is isomorphic to $S\cap N<S$. In particular it is a polycyclic group, hence of type $\FP$ (indeed
it is virtually a Poincar\'e duality group over $\Z$).
If $G$ has no $k$-torsion
 Fel{$'$}dmans's result may be applied to deduce that $\hd_{k}(G)=\hd_{k}(N\rtimes S)-\hd_{k}(S\cap N)=\hd_{k}(N\rtimes S)-h(S\cap N)$ and so in order to establish that $\hd_{k}$ coincides with Hirsch length for $G$, it suffices to do the same for $N\rtimes S$.

We shall need a variation on the Lennox--Robinson splitting theorem. The following is preparatory for this. The \emph{upper finite radical} of an abelian group is the join of the subgroups that are upper finite. An abelian group $A$ is called \emph{upper finite} if it has the property that $A/B$ is finite whenever it is finitely generated. The upper finite radical of an abelian group is upper finite. If $A$ is an abelian minimax group with upper finite radical $U$ then $A/U$ is finitely generated and free abelian. More generally, a minimax group is called \emph{upper finite} when all of its finitely generated quotients are finite: the class of upper finite groups is a radical class, meaning that the join of all subnormal upper finite subgroups of a group is upper finite. In a soluble minimax group $G$, the upper finite radical has least Hirsch length amongst normal subgroups $K$ such that the quotient $G/K$ is polycyclic. The following proposition uses the spirit of this definition: it can be interpreted as a ``near supplement'' result for the upper finite radical of a minimax group. Further near supplement and near complement results of this kind can be found in \cite{KK2013}. Theorem \ref{hd} can be interpreted as a statement about minimax groups $G$ in which the upper finite radical is virtually abelian.

\begin{proposition}\label{xxxxxxxx}
Let $G$ be a minimax group  with a normal subgroup $K$ such that $G/K$ is polycyclic. Then $G$ has a polycyclic subgroup $H$ such that  $HK$ has finite index in $G$.
\end{proposition}
\begin{proof} 
It suffices to construct a  polycyclic subgroup $H$ such that $HK$ has the same Hirsch length as $G$. For if $H$ is such a subgroup
then $HK/K$ is a subgroup of the polycyclic
group $G/K$, and having the same Hirsch length it must be of finite index. 

We focus on Hirsch length and proceed by induction on
 $h^{*}(G)$.
If $h^{*}(G)=0$ then $G$ is finite. Assume that $h^{*}(G)$ is non-zero. Let $A$ be an infinite subgroup with $h^{*}(A)$ as small as possible subject to $|G:N_{G}(A)|<\infty$. Replacing $G$ by a subgroup of finite index we may assume that $A$ is normal in $G$. By induction, there is a subgroup $L$ containing $A$ such that $L/A$ is polycyclic and $h(LK)=h(G)$. If $L$ has infinite index in $G$ then induction supplies a subgroup $H$ such that $H(K\cap L)$ has the same Hirsch length as $L$ and $HK$ now has the same Hirsch length as $G$. Therefore we may assume that for all choices of $A$, $G/A$ is polycyclic. 
If there is an infinite abelian torsion subgroup $A$ normal in $G$ such that $G/A$ is polycyclic then let $F$ be a finite subset of $G$ whose images in $G/A$ generate $G/A$ and let $P$ be the subgroup of $G$ generated by $F$. Since polycyclic groups are finitely presented, we deduce that $P\cap A$ is finitely generated as a $G$-operator group and therefore that $P\cap A$ has finite exponent. In this case $P$ is polycyclic and has the same Hirsch length as $G$ so we are done. Therefore we may assume that there are no torsion choices for $A$.

This implies that $G$ is just-non-polycyclic in the sense of Robinson and Wilson \cite{robinsonwilson} and our conclusion follows from the near-splitting result (\cite{robinsonwilson}, 2.4(iv)).
\end{proof}

\begin{proof}[Proof of Theorem \ref{hd}]
As in the proof of Theorem \ref{hd0}, 
we can reduce to the case where $G$ is finitely generated. 
Also we may assume that $G$ has a torsion-free abelian normal 
subgroup $A$ of finite Hirsch length such that $G/A$ is polycyclic. 
It follows that $A$ is a minimax group. 
Using the above splitting result and replacing $G$ by a subgroup of finite index, we may assume that $G=AS$ for some polycyclic subgroup $S$. 
As outlined in paragraph \ref{LR}, we may replace $G$ by the semidirect product $A\rtimes S$ in which the action of $S$ on $A$ is the same as its conjugation action within $G$, and with no loss of generality.

We shall need to appeal to the structure theory of torsion-free abelian minimax groups. If such a group has Hirsch length $h$, then
it is isomorphic to a subgroup of the additive group $\Z [1/m]^h$ for some natural number $m$.

Let $B<A$ be a free abelian subgroup of $A$ of rank $h$ and let $\phi$ be the endomorphism of $B$ given by $b\mapsto b^{m}$. Let $C$ denote the ascending HNN-extension $B_{B,\phi}$. The normal closure of $B$ in $C$ is the kernel of the obvious map from $C$ to the infinite 
cyclic group generated by the stable letter $t$ of this HNN extension. This kernel is isomorphic to $A\otimes \Z[1/m]$. Thus
we have a natural inclusion 
$$A\hookrightarrow C \cong \big( A\otimes \Z[1/m]\big) \rtimes \<t\>.$$
A key point to observe is that $C$ is torsion-free and constructible, so it is an inverse duality group over $\Z$ and Proposition \ref{constr}
applies. Indeed $C$ has a finite classifying space, namely the mapping torus of the endomorphism of the $h$-dimensional torus given by 
$x\mapsto x^m$.

The action of $S<G$ on $A$ by conjugation extends to an action on 
$A\otimes\Z[1/m]$ and then, 
by allowing the elements of $S$ to act trivially on the stable letter $t$, to $C$. This last action is well-defined because $\phi$ is central
in ${\rm{Aut}}(A\otimes \Z[1/m])$. We form the semidirect product $E = C\rtimes S$ and consider the commutative diagram relating $G$ and $E$.
\[\xymatrix{
A\ar@{>->}[r]\ar[d]&G\ar@{>>}[r]\ar[d]&S\ar@{=}[d]\\
C\ar@{>->}[r]&E\ar@{>>}[r]&S.
}\] 
Regarding Hirsch length, we have $h(C)=h(A)+1$ and hence $h(E)=h(G)+1$. Since $C$ is an inverse duality group and $\hd_{k}(S)=h(S)$ is finite, the hypotheses of Theorem \ref{feld}(ii) are satisfied and we obtain $\hd_{k}(E)=\hd_{k}(C)+h(S)$. Moreover, according to Proposition \ref{constr}, $\hd_{k}(C)=h(C)$. 

Finally, note that $h(G)$ and $\hd_k(G)$ remain unchanged if we replace $G=A\rtimes S$ by $\hat G = (A\otimes\Z[1/m])\rtimes S$.
This is because $\hat G$ is an ascending (directed) union of groups each of which contains $G$ as a subgroup of finite index.
Then $E$ is an extension of $\hat G$ by $\<t\>$ and so $\hd_{k}(E)\le\hd_{k}(\hat G)+1$.
Putting all this together we have
$\hd_{k}(G)\ge\hd_{k}(E)-1=\hd_{k}(C)-1+h(G/A)=h(C)-1+h(G/A)=h(A)+h(G/A)=h(G)$.
The reverse inequality $\hd_{k}(G)\le h(G)$ is well-known, see for example (\cite{bieriqmc}, Proposition 7.11). 
\end{proof}

\section{Filtered Colimits and Homological Dimension One}

In this part of the paper we fix a non-zero commutative ring $k$ and consider the following

\begin{conjecture}\label{mainII}
Every group $G$ of homological dimension one over $k$ is a filtered colimit of groups of cohomological dimension one over $k$.
\end{conjecture}

By a filtered colimit system, we mean a colimit system over a filtered category. A homological simplification arising from this is that the colimit functor $\colim$, when applied to abelian groups or modules, is exact and in particular, if $(G_{\lambda})$ is a filtered colimit system of groups and $(M_{\lambda})$ is a compatible system of modules, then the natural map
\[\colim H_{n}(G_{\lambda},M_{\lambda})\to H_{n}(\colim G_{\lambda},\colim M_{\lambda})\] is an isomorphism for all $n$. This has the important consequence:
\begin{lemma}\label{hcolim}
$\hd_{k}(\colim G_{\lambda})\le\sup_{\lambda}\hd_{k}(G_{\lambda})$.
\end{lemma}

We shall see that the conjecture holds in case $G$ is elementary amenable but that even in this situation, the groups involved in the colimit system may by necessity be non-amenable. We shall also see that the conjecture becomes false if dimension one is replaced by higher dimensions. 

Before beginning the analysis we review the spectral sequences and exact sequences associated with limits (i.e. inverse limits) and their derived functors. Limits arise naturally when considering filtered colimit systems of groups because group cohomology is contravariant as a functor of the group argument. While the colimit functor over a filtered category with values in an abelian category is exact, the limit functor $\lim$ has derived functors and so has a dimension which turns out to depend on the effective cardinality of the limit system.
These classical facts are found in Jensen's treatise, see especially (\cite{Jensen}, Th\'eor\`eme 4.2). In our context this may be stated as follows.
\begin{theorem}
Given a filtered colimit system of groups $(G_{\lambda})$ there is a first quadrant spectral sequence with $E^{p,q}_{2}=\lim^{p}H^{q}(G_{\lambda},M)$ converging to $H^{p+q}(\colim G_{\lambda}, M)$ for any module $M$ over the colimit $\colim G_{\lambda}$.
\end{theorem}

 In particular, if the colimit system is countably indexed then $\lim^{i}$ vanishes for $i>1$ and the spectral sequence simplifies to a set of short exact sequences
\[0\to\lim^{1}H^{n-1}(G_{\lambda},M)\to H^{n}(\colim G_{\lambda}, M)
\to\lim H^{n}(G_{\lambda},M)\to0.\] This provides a source of examples of countable groups for which $\cd_{k}=\hd_{k}+1$.

When the filtered colimit system is indexed by a set of cardinality $\aleph_{m}$ (for a natural number $m$) then $\lim^{i}$ vanishes for $i>m+1$. This provides a source of examples of groups of cardinality $\aleph_{m}$ for which $\cd_{k}=\hd_{k}+m+1$. 
We shall briefly review some of the history of such examples but first we record how the vanishing of $\lim^{i}$
 gives rise to the cohomological version of Lemma \ref{hcolim}:
\begin{lemma} \label{ccolim}
Let $\Lambda$ be a small filtered category of cardinality $\aleph_{m}$. Let $(G_{\lambda})$ be a filtered colimit system of groups over $\Lambda$. Then 
$\cd_{k}(\colim G_{\lambda})\le\sup_{\lambda}\cd_{k}(G_{\lambda})+m+1$.
\end{lemma}

\medskip

For the next three subsections, we work entirely with the coefficient ring $\Z$.

\subsection{Torsion-free nilpotent groups of finite rank}\ 

Gruenberg used the $\lim$-$\lim^{1}$ methodology to show that the cohomological dimension of a torsion-free nilpotent $G$ is $h(G)$ or $h(G)+1$ according to whether 
 $G$ is finitely generated or not finitely generated. See (\cite{gruenberg}, \S8.8, Theorem 5).

\subsection{A finitely generated metabelian group with $\hd=3$ and $\cd=4$}\label{sse:34} 
An alternative approach to the $\lim$-$\lim^{1}$ method uses the universal coefficient theorem in case $k=\Z$. This says that for a $\Z G$-module $M$ with trivial $G$-action there are short exact sequences
\[0\to\ext(H_{n}(G,\Z),M)\to H^{n+1}(G,M)\to\hom(H_{n+1}(G,\Z),M)\to0.\] In particular if $G$ is a countable 
 group of homological dimension $n$ and $H_{n}(G,\Z)$ is not a free abelian group then $\cd_{\Z}G=n+1$. Bieri used this argument to give an example of a group of cohomological dimension 4 and homological dimension 3 in his paper \cite{bieri1972}. His example is the split exension of the additive group $\Q^{2}$ by an infinite cyclic group $\Z$ with action $n\cdot(a,b)=(n^{-1}a,nb)$: the point here is that the group is constructed in order to ensure that $H_{3}(G,\Z)\iso\Q$. 

\subsection{A finitely generated metabelian group with $\hd=2$ and $\cd=3$}\label{HH} The construction of \S\ref{sse:34} cannot be used to address the case cohomological dimension 3 and homological dimension 2.
So
work continued on soluble groups and the matrix group 
\[H=
\left\{\begin{pmatrix}
(\tfrac23)^{j}&0\\\ell&1
\end{pmatrix}
;\ j\in\Z,\ell\in\Z[\tfrac16]\right\}
\]
emerged as a key example because for a period of time it remained a question whether this group had cohomological  
dimension 2 or 3. Gidlenhuys settled this by showing that the cohomological dimension is indeed $3$ for this group, see \cite{gildenhuys}. 
One can see that $\hd(H)=2$ by observing that $H\cong \Z[1/6]\rtimes\Z$, so in particular $H$ is locally-cyclic by cyclic. Alternatively
one could observe that it has Hirsch length $2$ and appeal to the results in Part I.  The group $H$ is the colimit that one obtains in the
following proposition when $(p,q)=(2,3)$.

\begin{proposition}
Let $p$ and $q$ be coprime integers neither of which equals $0$, $1$, or $-1$. Then $B(p,q):=\langle a,b;\ b^{-1}a^{p}b=a^{q}\rangle$ has cohomological dimension $2$ and admits a surjective homomorphism $\phi$ such that the colimit $G$ 
 of the sequence 
 $B(p,q)\buildrel\phi\over\to B(p,q)\buildrel\phi\over\to B(p,q)\buildrel\phi\over\to\dots$ is metabelian and 
 $\cd_\Z(G)=3$, but $\hd_\Z(G)=2$.  
 \end{proposition} 
 
 \begin{proof}
 First note that $B:=B(p,q)$ has cohomological dimension $2$ because it is
 an HNN extension of the infinite cyclic group $\<a\>$,
 indeed
 it has a finite 2-dimensional classifying space  obtained by attaching a
 cylinder to a circle, with one end of the cylinder wrapping around $p$
 times, and the other $q$ times. ($B$ cannot be of cohomological dimension $1$ because it can be shown to be torsion-free with one end. Every $2$-generator torsion-free group with more than one end is free.) 
 
 The assignments $a\mapsto a^{pq}$ and $b\mapsto b$ clearly define an endomorphism of $B$. This is surjective because $\langle a^{pq},b\rangle$ contains both $ba^{pq}b^{-1}=a^{p^{2}}$ and $b^{-1}a^{pq}b=a^{q^{2}}$, and therefore $a$. To see that the resulting colimit $G$ is metabelian, 
 observe that for any word $w$ in $a$ and $b$ with exponent sum $k$ in $b$,  there is a sufficiently large natural number $n$ such that 
 $\phi^{n}(w)$ 
 lies in the coset $\langle a\rangle b^{k}$. Let $u$ and $v$ be elements of the derived subgroup of $B$. Then $u$ and $v$ have exponent sum zero in $b$ and therefore $\phi^{n}(u)$ and $\phi^{n}(v)$ both belong to 
 $\langle a\rangle$ and so commute.
   
 Gildenhuys (\cite{gildenhuys}, Theorem 5) characterizes the soluble groups that have cohomological dimension at most two over $\Z$.
 In particular he shows that if such a group is finitely generated then it is finitely presentable. Our colimit $G$,
 on the other hand, is not finitely presentable, because the maps in the
 directed system defining it all have
 non-trivial kernel: the normal form theorem for HNN extensions assures us
 that
 $[a,b^{-1}ab]\in\ker \phi$ is non-trivial, for example. Thus the colimit
 must have dimension at least $3$. But 
 since homology commutes with directed colimits, the homological dimension
 of $G$ is $2$.
\end{proof}

Note that although the group $G$ constructed in the preceding proof is not finitely presented, the group $\hat G = G\rtimes_{\overline{\phi}}\Z$
is, where $\overline{\phi}:G\to G$ is the isomorphism induced by $\phi: B\to B$. Indeed $\hat G$ has a balanced presentation,
$$
\hat G = \langle a, b, t ; \, b^{-1}a^pb = a^q,\ t^{-1}at=a^{pq},\, t^{-1}bt=b\rangle.
$$
Thus we obtain (concise) finite presentations of metabelian groups that have homological dimension $3$ and cohomological dimension $4$.

As a variation on this theme, we shall prove that the group $G_1$ with presentation
\[
G_{1}:=\langle a,b,c,d;\ b^{-1}a^{2}b=a^{3}, c^{-1}a^{3}c=a^{5}, d^{-1}a^{5}d=a^{2}, bc=cb, cd=dc, db=bd\rangle
\]
is metabelian.
Similar reasoning to that above shows that the assignments $a\mapsto a^{30}$, $b\mapsto b$, $c\mapsto c$, $d\mapsto d$ determine a surjective endomorphism $\theta$ of $G_{1}$ such that the colimit 
$\colim(G_{1}\buildrel\theta\over\to
G_{1}\buildrel\theta\over\to
G_{1}\buildrel\theta\over\to\dots)$ is metabelian. But now the deep theory of Bieri--Strebel characterizing finitely presented metabelian groups may be brought into play. This metabelian colimit has $2$-tame Bieri--Neumann--Strebel invariant and is therefore finitely presented. It follows that the colimit process must stop in a finite number of steps and hence $G_{1}$ itself is metabelian and, strikingingly,  $\theta$ is an automorphism. Conceivably this example, although not central to the directions of the present paper, may have interest elsewhere.

\subsection{Homological Dimension One}\label{hdo}

Cohomological dimension one is completely understood by the theorem \cite{dunwoody1979} of Dunwoody: a group $G$ has cohomological dimension $\le1$ over $k$ if and only if it is the fundamental group of a graph of finite groups whose orders are units in $k$. The case where $k=\Z$ was proved earlier by Stallings (for finitely generated groups) and Swan (in general) and simply asserts that the groups are free. Homological dimension one is not understood. The following Lemma provides the only known source of such groups:
\begin{lemma}\label{fc} Let $n$ be a natural number.
If $G$ is the filtered colimit of a family of groups $H_\lambda$ with $\hd_kH_\lambda\le n$ then $\hd_k G\le n$.
\end{lemma}
\begin{proof}
Homology commutes with filtered colimits:
$$H_m(G,M)\iso\colim H_m(H_\lambda,M)$$ for all $m$, and the result follows.
\end{proof}
This leads naturally to Conjecture \ref{mainII}.
One direction of the conjecture is true and easily deduced from Lemma \ref{fc}: if $G$ is the filtered colimit of groups of cohomological dimension $\le1$ the Lemma shows that $G$ has homological dimension $\le1$. The other direction 
 is open even with $k=\Z$, in which case the conjecture takes a simplified form:

\begin{conjecture}\label{B}
Every group of homological dimension $\le1$ over $\Z$ is locally free.
\end{conjecture}

In \cite{krophollerlinnelllueck} it is shown that any counterexample to this would also be a couterexample to the Atiyah conjecture which states that the von Neumann dimensions of Hilbert modules for torsion-free groups are integer valued whenever they are finite.

\begin{proof}[Proof that Conjecture \ref{mainII} implies Conjecture \ref{B}]
Let $G$ be a group of homological dimension 1 over $\Z$.
According to Stallings--Swan, groups of cohomological dimension $\le1$ are free, so Conjecture \ref{mainII} implies that $G$ is a filtered colimit $\displaystyle\colim G_{\lambda}$  
of free groups. Let $S$ be a finite subset of $G$. Choose $\lambda$ so that the image of $G_{\lambda}$  in $G$ contains $S$. Choose a finitely generated $H_{\lambda}\subseteq G_{\lambda}$ such that the image of $H_{\lambda}$ in $G$ contains $S$. Now for each $\mu\ge\lambda$ we have that the image $H_{\mu}$ of $H_{\lambda}$ in $G_{\mu}$ is free and either the connecting map is an isomorphism or the rank of $H_{\mu}$ is strictly less than that of $H_{\lambda}$. Since the ranks are finite, eventually we will reach a $\mu$ where the rank is minimum and the image of this $H_{\mu}$ in $G$ will be an isomorphic copy of $H_\mu$, thus providing a free group containing $S$.
\end{proof}

Note that the deduction of Conjecture \ref{B} from \ref{mainII}   makes use of the fact that free groups of finite rank are Hopfian in a very strong sense: namely that proper quotients are either non-free or have fewer generators. This constrasts with the situation over $\Q$ where there exist groups of cohomological dimension 1 with two generators and many quotients that also have cohomological dimension 1 and two generators: this is illustrated in the proof of Proposition \ref{34} below.  

Although one might be tempted to think that groups of homological dimension one are locally of cohomological dimension 1, this is not the case. The restricted wreath product $C_p\wr C_\infty$ has homological dimension 1 over any $k$ in which $p$ is invertible and it is finitely generated of cohomological dimension 2. However it is not a counterexample to Conjecture \ref{mainII}. Indeed, we can confirm Conjecture \ref{mainII} for elementary amenable groups.

\begin{proposition}\label{34}
Every elementary amenable group of homological dimension $\le1$ over $k$ is a filtered colimit of groups of cohomological dimension $\le1$ over $k$.
\end{proposition}
\begin{proof}
Suppose that $G$ is a finitely generated elementary amenable group with $\hd_k(G)=1$. Let $T$ be the largest normal locally finite subgroup of $G$. Then $G$ has Hirsch length $1$ and $G/T$ is either infinite cyclic or infinite dihedral. We consider the two cases.
\begin{enumerate}
\item $G/T$ is cyclic: Let $g$ be a generator of $G$ modulo $T$ and let $F_n$ ($n\ge0$) be a chain of finite subgroups with union $T$. For each $n$ let $B_n$ be the (finite) subgroup generated by $F_n$ and $g^{-1}F_ng$. Then we can build the HNN-extensions $H_n:=B_n*_{F_n,g}$. There are natural maps $H_n\to H_{n+1}$ induced by the inclusions $F_n\to F_{n+1}$ and $G$ is the colimit of the $H_n$. Notice that each $H_n$ is the fundamental group of a graph of finite groups and so has cohomological dimension one. 
\item $G/T$ is dihedral: This time there are two subgroups $H$ and $K$ which contain $T$ as a subgroup of index $2$ and $G$ is the free product with amalgamation $H*_TK$. Both $H$ and $K$ are locally finite and so we can view $G$ as the colimit of a sequence of free products with amalgamation of finite groups by choosing chains of finite subgroups in $H$ and $K$ in the same spirit as the first case.
\end{enumerate}
In general, $G$ may be viewed as the filtered colimit of its finitely generated subgroups and combining this with the analysis above gives the desired conclusion.
\end{proof}

Two natural questions arise in this context and both have negative answers:

\begin{question}
If $\ell$ is a natural number, is it true that every group of homological dimension $\ell$ is a filtered colimit of groups of cohomological dimension at most $\ell$? In other words, does Conjecture \ref{mainII} remain plausible if the bound $1$ is replaced by a higher bound $\ell$?
\end{question}
\begin{proof}[Discussion]
No. If $\ell=3$ we have the following decisive counterexample with coefficient ring $\Z$. Let $D$ be the matrix group
 \[
 \left\{
 \begin{pmatrix} 2^{r}3^{s}&0\\q&1\end{pmatrix};\ r,s\in\Z,\ q\in\Z[\tfrac16]
 \right\}.
 \] 
 Let $H$ be the subgroup generated by $\begin{pmatrix} \frac23&0\\0&1\end{pmatrix}$ and
 $\begin{pmatrix} 1&0\\1&1\end{pmatrix}$ and let $G:=D*_{H,t}$ be the HNN-extension in which the stable letter centralizes $H$. Thus $G$ has the finite presentation
 $$\langle a,b,c,t;\ b^{-1}ab=a^{2},\ c^{-1}ac=a^{3},\  bc=cb,\ 
 tbc^{-1}=bc^{-1}t,\ ta=at\rangle.$$

 Moreover, $H$ is the matrix group of (\ref{HH}); 
 we saw that it has homological dimension $2$ and cohomological dimension $3$. 
Similarly, noting that $D\cong \Z[1/6]\rtimes\Z^2$, we see that $D$ has homological dimension $3$ and cohomological dimension at most $4$. 
It follows by consideration of the Mayer-Vietoris sequences associated to the HNN description of $G$ that $G$ has homological
dimension $3$ and cohomological dimension at most $4$. Finally, since
  $G$ contains the direct product $H\times\langle t\rangle$, which has cohomological dimension $4$,
we conclude that  $G$ has cohomological dimension at most $4$. The conclusion: $G$ is a finitely presented group of cohomological dimension $4$ and homological dimension $3$. Finite presentation prevents $G$ from being expressible as a filtered colimit of groups that do not already contain subgroups isomorphic to $G$ itself and hence any attempt to express $G$ as a filtered colimit involves groups of cohomological dimension $4$.
 
 This leaves a question about the case $\ell=2$. We do not know of an example over $\Z$ to rule out the possibility that every group of homological dimension $2$ over $\Z$ is a filtered colimit of groups of cohomological dimension $2$ over $\Z$. However there are counterexamples if one works over $\Q$. For example, consider the function field $\F_{p}(x)$ in one variable over the finite prime field $\F_{p}$ and consider the subgroup of $GL_{2}(\F_{p})$ generated by the matrices
\[
\begin{pmatrix}
x&0\\0&1 
\end{pmatrix},\ 
\begin{pmatrix}
x+1&0\\0&1 
\end{pmatrix},\ 
\begin{pmatrix}
1&0\\1&1 
\end{pmatrix}.
\]
Again this group is known to be finitely presented and can be expressed as an ascending HNN extension  
with base the lamplighter group  $C_{p}\wr C_{\infty}$  generated by the first and third of the matrices above. 
\end{proof}

\begin{question}
Returning to the case $\ell=1$, is it conceivable that every elementary amenable group of homological dimension one is a filtered colimit of elementary amenable groups of cohomological dimension one?
\end{question}
\begin{proof}[Discussion]
No. The lamplighter wreath product $W=C_{p}\wr C_{\infty}$ is already a counterexample to this proposal. Dunwoody's classification shows the elementary amenable groups of cohomological dimension one are either locally finite or virtually cyclic, and such a group cannot map onto $W$.
\end{proof}

\section{Cohomological Dimension}

The inequality $\hd_{k}\le\cd_{k}$ holds universally and so we only need to consider groups with finite homological dimension in this section. For elementary amenable groups there is a natural conjecture for cohomological dimension.

\begin{conjecture}\label{mainIII}
Let $k$ be a non-zero commutative ring and let $G$ be an elementary amenable group with no $k$-torsion. Then $\cd_{k}(G)$ is finite if and only if $G$ has cardinality less than $\aleph_{\omega}$. Moreover,
\begin{enumerate}
\item
if $G$ is constructible then $\cd_{k}(G)=h(G)$;
\item
if $G$ is countable but not constructible then $\cd_{k}(G)=h(G)+1$; and
\item
if $G$ is uncountable of cardinality $\aleph_{n}$, ($0<n<\omega$), then $\cd_{k}(G)=h(G)+n+1$.
\end{enumerate}
\end{conjecture}

Roughly speaking Conjecture \ref{mainIII} amounts to saying that the cohomological dimension is as  
 large as it could possibly be given the constraints laid down by basic inequalities.
In case $h(G)=0$ (meaning that $G$ is locally finite) the statement here was conjectured by Holt \cite{holt1981} for finite prime fields $k$ and the consistency of this statement with Zermelo--Fraenkel set theory was established by Kropholler and Thomas \cite{krophollerthomas}. In case $G$ is countable the result has been established for $k=\Q$ when there is a bound on the orders of the finite subgroups of $G$. The general case of the conjecture involves a melding of these two cases and remains stubbornly difficult to prove. We shall make some reductions and examine one special case in greater detail.

\subsection{Reducing to the minimal counterexamples} 

Supposing the conjecture to be false, let us see how it might fail. 
If $G$ is a counterexample then it is natural to look at sections of $G$ and see whether these are also counterexamples. By a \emph{section}, we mean any group of the form $S/T$ where $T$ is normal in $S$ and $S$ is a subgroup of $G$. The homological dimension and Hirsch length of any counterexample are finite while the cohomological dimension is less than predicted. Consider first the case of countable groups. We have the following reduction in this case.

\begin{lemma}
If $G$ is a countable counterexample to Conjecture \ref{mainIII} then every finitely generated subgroup $H$ of $G$ with $h(H)=h(G)$ is either constructible or is a counterexample to one of the Conjectures \ref{hdconj}, \ref{mainIII}.
\end{lemma}
\begin{proof}
Let $G$ be such a counterexample. 
Suppose that no finitely generated subgroups $H$ with $h(H)=h(G)$ are counterexamples to Conjecture \ref{hdconj}. Then $\hd_{k}(H)=h(H)$ for all such subgroups.
We deduce that $G$ is not constructible but nevertheless $\hd_{k}(G)=\cd_{k}(G)$.
Suppose that $H$ is a finitely generated subgroup of $G$ with $h(H)=h(G)$. Suppose that $\hd_{k}(H)=h(H)$. Then $\hd_{k}(H)=\hd_{k}(G)$ from which it follows that $\hd_{k}(H)=\cd_{k}(H)$ and therefore, assuming Conjecture \ref{mainIII} is valid for $H$, we deduce that $H$ is constructible. \end{proof}

In view of this it is natural to consider finitely generated groups.

\begin{proposition}
If $G$ is a finitely generated counterexample to Conjecture \ref{mainIII} and $T$ is the largest 
locally finite normal subgroup of $G$ then either $G/T$ is also a counterexample or $T$ is infinite and 
$G/T$ is constructible.
\end{proposition}
\begin{proof}
Replacing $G$ by a subgroup of finite index we may assume that $G/T$ is torsion-free. We then have $\cd_{k}(G/T)\le\cd_{k}(G)=h(G)=h(G/T)$ and it follows that $G/T$ is either constructible or is again a counterexample. If $G/T$ is constructible and $T$ is finite then $G$ is also constructible and so is no counterexample. Hence, $T$ must be infinite in case $G/T$ is constructible.
\end{proof}

In order to make progress we need to restrict attention to cases where Conjecture \ref{hdconj} holds. Therefore we shall confine attention to the case $k=\Q$ for the remainder of this paper.

\subsection{Nilpotent-by-polycyclic-by-finite groups over $\Q$}

In the remainder of the article we attempt to prove Conjecture \ref{mainIII} for
nilpotent-by-polycyclic-by-finite groups $G$ with cardinality $|G|<\aleph_\omega$ and with $k=\Q$.
We succeed in the case of countable groups. In the case of groups of greater cardinality
we reduce the problem to a conjecture about the first cohomology of countable, locally
finite abelian groups.

\subsection{The countable case}

We shall need the following elementary argument from commutative algebra.

%

\begin{lemma}\label{somecommutativeringtheory}
Let $S$ be a commutative ring and let $U$ and $V$ be $S$-modules both of which admit composition series. 
Let $\mathcal I$ be the set of maximal ideals $I$ of $S$ such that $S/I$ is isomorphic to a composition factor of $U$ and let $\mathcal J$ be the set of maximal ideals $J$ of $S$ such that $S/J$ is isomorphic to a composition factor of $V$. If $\mathcal I\cap\mathcal J=\emptyset$ then $\ext_S^n(U,V)=0$ for all $n\ge0$.
\end{lemma}
\begin{proof}
The long exact sequence for $\ext$ may be used to reduce to the case when both $U$ and $V$ are irreducible. The hypothesis  $\mathcal I\cap\mathcal J=\emptyset$ is then reduced to the assertion that the annihilators of $U$ and $V$ are distinct maximal ideals $I$ and $J$ of $S$. Since $S$ is commutative, all the $\ext$ groups $\ext_S^n(U,V)$ inherit $S$-module structures and are annihilated by both $I$ and $J$ and therefore by $S=I+J$.
\end{proof}

\begin{theorem}\label{X0case}
Let $G$ be a countable group that is nilpotent-by-polycyclic-by-finite. 
Then $G$ has finite rational cohomological dimension if and only if $\hdq(G)<\infty$, in which
case
\begin{enumerate}
\item $\cdq(G)=\hdq(G)$ if $G$ is constructible, and 
\item $\cdq(G)=\hdq(G)+1$ if $G$ is not constructible.
\end{enumerate}
\end{theorem}

\begin{proof}
As in (\cite{gildenhuysstrebel1981}, \S1.2) we have
$$h(\overline G)=\hd_{\Q}(\overline G)\le\cd_{\Q}(\overline G)\le h(\overline G)+1$$ for any quotient $\overline G$ of $G$ (including $G$ itself) and the condition $\cd_{\Q}(\overline G)=h(\overline G)$ is inherited by quotients of $\overline G$. If $G$ is constructible then we know that $\cd_{\Q}(G)=\hd_{\Q}(G)$ and there is nothing further to prove. Therefore we may assume that $G$ is not constructible.
Let $N$ be a normal nilpotent subgroup of $G$ such that $G/N$ is polycyclic-by-finite. Since $G$ is not constructible it follows that $G/[N,N]$ is also not constructible and by the above remarks we may replace $G$ by this quotient and so assume that $G$ is abelian-by-polycyclic-by-finite. If $G$ is locally polycyclic-by-finite and infinitely generated then the methods of \cite{gildenhuysstrebel1981} can be used to show that $\cd_{\Q}(G)=\hd_{\Q}+1$. So we assume now that $G$ is finitely generated and without loss of generality we may replace $G$ by a subgroup of finite index and so assume that $G$ has an abelian normal subgroup $N$ such that $G/N$ is polycyclic. In this case, $N$ is finitely generated as a $G$-operator group and may be regarded as a finitely generated $\Z Q$-module where $Q=G/N$. 

If $N/N^{p}$ is finite for all primes $p$ then classical arguments of Philip Hall show that $N$ has finite rank: this reduces to the minimax case which is completely resolved in \cite{krophollermartinezpereznucinkis}. So we focus on the case when there is a prime $p$ for which $T:=N/N^{p}$ has infinite rank. Again, we may pass to the quotient $G/N^{p}$ and so we assume that $T$ is an elementary abelian $p$-subgroup which is normal in $G$ and that $G/T$ is polycyclic. We could now go further and pass to a just-non-polycyclic quotient when it would become possible to invoke the analysis of \cite{robinsonwilson}. However we have found that the main obstacle in proceeding further lies in dealing with the possibility that $G$ is a non-split extension of $N$ by $G/N$ and such a non-split extension may persist even in the just-non-polycyclic case. 

We assume henceforth that our group $G$ is \emph{primitive} and just-non-polycyclic. 
The notion and theory of primitive just-non-polycyclic groups is to be found in \cite{robinsonwilson}.
We write $T$ for the Fitting subgroup of $G$ and we suppose that $G/T$ is the split extension $Q:=P\rtimes A$ of two free abelian groups $P$ and $A$ of finite rank. We may also suppose that $P$ is a \emph{plinth} in the group $Q$. The terminology plinth was introduced by Roseblade \cite{roseblade}. 

Our goal now is to prove that $\cd_{\Q}(G)\ge h(G)+1$. We shall do this by constructing a $\Q G$-module $M$ such that $H^{d+1}(G,M)$ is non-zero, where $d$ denotes the Hirsch length of $G$. We build $M$ as a direct sum of an infinite sequence of modules $(M_{n};\ n\ge1)$ where each $M_{n}$ is induced from a $\Q T$-module that has finite dimension as a $\Q$-vector space. In computing cohomology we take advantage of the fact that polycyclic groups satisfy Poincar\'e duality. We may pass to a subgroup of finite index if necessary in order to assume that $Q$ is both torsion-free and orientable as a Poincar\'e duality group. Thus we have isomorphisms $H^{j}(Q,V)\iso H_{d-j}(Q,V)$ for arbitrary $\Q Q$-modules $V$. Note that $G$ and $Q$ have the same Hirsch length $d$ because $T$ is torsion. We only need the extreme case of duality, namely
$H^{d}(Q,V)\iso H_{0}(Q,V)$ for all $V$. We shall choose the sequence of $\Q G$-modules $M_{n}$ so that
\begin{enumerate}
\item
$H^{*}(T,M_{n})$ vanishes;
\item
$H^{1}(T,\bigoplus M_{n})$ is non-zero.
\end{enumerate}
In addition, we shall make sure that we have some control over some of the non-zero elements of $H^{1}(T,\bigoplus M_{n})$. Since countable locally finite groups have cohomological dimension one over $\Q$, we know that $H^j(T,\blah)$ vanishes when $j>1$. Also, condition (i) above will imply that for our chosen sequence, $H^{0}(G,\bigoplus M_{n})=0$. A spectral sequence corner argument together with Poincar\'e duality then yields 
$$H^{d+1}(G,M)\iso H^{d}(Q,H^{1}(T,M))\iso H_{0}(Q,H^{1}(T,M)).$$
The short exact sequence
\[
0\to\bigoplus M_{n}\to\prod M_{n}\to\frac{\prod M_{n}}{\bigoplus M_{n}}\to0
\]
leads to a long exact sequence in the cohomology of $T$ which reduces in our context to an isomorphism
\[
H^{1}\left(T,\bigoplus M_{n}\right)\iso H^{0}\left(T,\frac{\prod M_{n}}{\bigoplus M_{n}}\right).
\]
For a $T$-module $V$, we write $V^{T}$ for the subspace of $T$-fixed points: that is $V^{T}=H^{0}(T,V)$. So our goal is now to make choices of the $M_{n}$ in such a way that 
\[
H_{0}\left(Q,\left(\frac{\prod M_{n}}{\bigoplus M_{n}}\right)^{T}\right)
\]
is non-zero.

\subsection*{Application of a theorem of Roseblade} In order to choose suitable modules $M_n$ we first need a set $\mathcal S$ of subgroups $S$ of $T$ which are normal in $G$, of arbitrarily large finite index in $T$, and are such that the short exact sequences $T/S\mono G/S\epi G/T$ split. We shall use an argument of Roseblade \cite{roseblade} to achieve this.

Recall that $P$ is a plinth in the group $Q=PA$. Also, $T$ is an elementary abelian $p$-group for some prime $p$ and the action of $G$ by conjugation on $T$ makes $T$ into an $\F_pQ$-module, since $Q=G/T$.
Roseblade's (\cite{roseblade}, Theorem E) shows that if $\lambda$ is a non-zero element of the group ring $R=\F_pP$ then there is a maximal ideal $J$ of $R$ which contains no conjugate of $\lambda$. The proof of (\cite{roseblade}, Theorem E) in fact shows that there are infinitely many such $J$ and the quotient fields $\F_pP/J$ are of unbounded (finite) cardinality. Let $\mathcal J$ denote the set of all these maximal ideals.

The Robinson--Wilson theory \cite{robinsonwilson} shows that $T$ is torsion-free of finite rank as an $\F_pP$-module. By Roseblade's (\cite{roseblade}, Theorem C) there is a free $R$-submodule $U$ of $T$ and a non-zero ideal $\Lambda$ of $R$ such that every finitely generated $R$-submodule of $T/U$ is annihilated by a product $\Lambda^{x_1}\Lambda^{x_2}\cdots\Lambda^{x_n}$ of conjugates of $\Lambda$ under $Q$. 
Now if $J$ belongs to $\mathcal J$ then 
$J+\Lambda^{x_1}\Lambda^{x_2}\cdots\Lambda^{x_n}=R$ for any choice of finite product of conjugates of $\Lambda$. It follows that $TJ\cap U=UJ$ because any element $y$ of $TJ\cap U/UJ$ is annihilated by some product $\Lambda^{x_1}\Lambda^{x_2}\cdots\Lambda^{x_n}$ and hence there is an expression $1=j+z$ with $zJ\subseteq UJ$ and $j\in J$. Thus $y=yj+zj\in UJ$. This shows that $TJ<T$ for any such ideal $J$. Since $T$ has finite rank, we deduce that $T/TJ$ is both finite and non-trivial.

We appeal to Lemma \ref{somecommutativeringtheory} with $S:=\F_pP$ to deduce that $H^n(P,T/TJ)=\ext_{\F_pP}^n(\F_p,T/TJ)=0$ for all $n$ and all $J\in\mathcal{J}$ other than the augmentation ideal $\mathfrak p$. Let $J$ be a member of $\mathcal{J}\setminus\{\mathfrak p\}$ and let $J_0$ be the intersection of the conjugates of $J$. Then we still have $H^n(P,T/TJ_0)=0$ for all $n$. The ideal $J_0$ of $\F_pP$ is $Q$-invariant, so $TJ_0$ is an $\F_pQ$-submodule. A spectral sequence argument shows at once that $H^n(Q,T/TJ_0)=0$ for all $n$.
In the group $G$, $TJ_0$ is a \emph{normal} subgroup and we can form the quotient $G/TJ_0$ which is an extension of $T/TJ_0$ by $Q$. Our vanishing cohomology result shows that this extension is split. 

In conclusion, we now know that there are normal subgroups $S:=TJ_0$ of $G$ having finite index in $T$ such that $G/S$ is the split extension of $T/S$ by $Q$ and such that the factors $T/S$ are finite of unbounded order. 
These $S$ populate our set $\mathcal S$.

\subsection{A choice of finite subgroups exhausting $T$}

Choose an ascending chain $F_{1}<F_{2}<\dots<F_{j}<\cdots$ of finite subgroups of $T$ such that $T=\bigcup_{j}F_{j}$. Since the subgroups of 
$\mathcal S$ have unbounded index we can find, for each $n$, a subgroup $T_{n}\in\mathcal S$ of index greater than $|F_{n}|$, and therefore such that $T_{n}+F_{n}<T$.

\subsection*{The definition of $M_{n}$}
For each $n$, let $M_{n}$ be the augmentation ideal of the rational group algebra $\Q[T/T_{n}]$. The action of $Q$ on $T$ by conjugation induces an action of $Q$ on $T/T_{n}$ and this extends to an action of the split extension $T/T_{n}\rtimes Q$ on $M_{n}$. 
Since the extension $T/T_{n}\mono G/T_{n}\epi Q$ is split, we may regard this as an action of $G/T_{n}$ and through the natural surjection $G\to G/T_{n}$ we obtain an action of $G$ on $M_{n}$. 

Notice that each $M_{n}$ is finite dimensional as a $\Q$-vector space.  
Since $T$ is locally finite, it follows that each $M_{n}$ is injective as a $\Q T$-module. [To see this note first that every $\Q T$-module $V$ is flat since it is the direct limit of the projective modules $V\otimes_{\Q T_{n}}\Q T$. The dual $V^{*}:=\hom_{\Q}(V,\Q)$ is therefore injective and again finite dimensional. It follows that for any $V$, the double dual of $V$ is injective and so in particular if $V$ is finite dimensional then $V$ itself is injective, being isomorphic to $V^{**}$.] 

By setting $M_{n}$ equal to the augmentation ideal in $\Q[T/T_{n}]$ we have arranged that we have $M_{n}^{T}=0$ for all $n$. Being also injective, the modules $M_{n}$ are cohomologically acyclic as asserted in (i) above. 

\subsection*{A sequence $(w_{n})$ of elements of $M_{n}$} The following argument is essentially the same as that in (\cite{Kropholler1985}, Proposition 2.4); for the reader's convenience we provide careful details.

As a $\Q$-vector space, $\Q[T/T_{n}]$ has a basis in bijective correspondence with the elements of $T/T_{n}$ and we write $e_{t}$ for the basis vector which naturally corresponds to $t\in T/T_{n}$. The vectors $u_{n}:=\sum_{t\in T/T_{n}}e_{t}$ and $v_{n}:=\sum_{t\in T_{n}+F_{n}/T_{n}}e_{t}$ both belong to $\Q[T/T_{n}]^{F_{n}}$. Therefore the vector $w_{n}:=u_{n}-|T:T_{n}+F_{n}|v_{n}$ is an element of $\Q[T/T_{n}]^{F_{n}}$. Moreover, the augmentation map determined by $e_{t}\mapsto1$ evaluates to zero on $w_{n}$ so we have
\begin{enumerate}
\item[(iii)]
$w_{n}\in M_{n}\cap\Q[T/T_{n}]^{F_{n}}=M_{n}^{F_{n}}$.
\end{enumerate}
The action of $Q$ by conjugation on $T$ induces an action on $M_{n}$ which permutes the basis $(e_{t}-e_{1};\ 1\ne t\in T/T_{n})$. The map $\mu:M_{n}\to\Q$ defined by $e_{t}-e_{1}\mapsto1$ is a $\Q Q$-module homomorphism to the trivial module. Under this map we find that

$w_{n}=\sum_{1\ne t\in T/T_{n}}(e_{t}-e_{1})-
|T:T_{n}+F_{n}|\sum_{1\ne t\in T_{n}+F_{n}/T_{n}}(e_{t}-e_{1})$ maps to
$|T:T_{n}|-1-|T:T_{n}+F_{n}|(|T_{n}+F_{n}:T_{n}|-1)
=|T:T_{n}+F_{n}|-1$ which is non-zero because of the strict inclusion
$T_{n}+F_{n}<T$. By universality, $\mu$ factors through $H_0(Q,M_n)$ and hence 
\begin{enumerate}
\item[(iv)]
the image of $w_{n}$ under the natural map $M_{n}\to H_{0}(Q,M_{n})$ is non-zero.
\end{enumerate}
The sequence $(w_{n},\ n\in\N)$ is an element of the product
$\prod M_{n}$ with the property that for all $t\in T$, it differs from the translate $(w_{n}t)$ in only finitely many coordinates. 
[If $t\in F_{n_{0}}$ then $w_{n}t=w_{n}$ for all $n\ge n_{0}$, and since the chain of $F_{j}$ exhausts $T$, each $t$ is to be found in some $F_{n_{0}}$.] 
Therefore $(w_{n})+\bigoplus M_{n}$ is an element of $\left(\tfrac{\prod M_{n}}{\bigoplus M_{n}}\right)^{T}$ which has non-zero image in 
$H_{0}\left(Q,\left(\tfrac{\prod M_{n}}{\bigoplus M_{n}}\right)^{T}\right)$. The diagram in Figure \ref{fig1} illustrates what is going on and is explained in its caption.

\begin{figure}[htbp]
\[
\xymatrix
{
&M\ar@{>>}[r]\ar@{>->}[d]
&H_{0}(Q,M)\ar[r]\ar@{>->}[d]
&\bigoplus\left(H_{0}(Q,M_n)\right)\ar@{>->}[d]\\
&M^*\ar@{>>}[r]^{\chi}\ar@{>>}[d]
&H_{0}(Q,M^*)\ar[r]^{\psi}\ar@{>>}[d]
&\prod (H_{0}(Q,M_n))\ar@{>>}[d]^{\omega}\\
(w_{n})\in\prod M_n^{F_n}\ar[ur]^{\phi}\ar[r]\ar@{.>}[dr]_{\alpha}
&\dfrac{M^*}{M}\ar@{>>}[r]
&H_{0}\left(Q,\tfrac{M^*}{M}\right)\ar[r]^{\delta}
&\dfrac{\prod \left(H_{0}(Q,M_n)\right)}{\bigoplus \left(H_{0}(Q,M_n)\right)}\\
&\left(\dfrac{M^*}{M}\right)^T\ar@{>>}[r]_{\beta\blah\blah}\ar@{>->}[u]
&H_{0}\left(Q,\left(\tfrac{M^*}{M}\right)^T\right)\ar[u]^{\gamma}\\
}
\]
\caption{A diagrammatic illustration of the proof of Theorem III.8:
$M$ denotes $\bigoplus M_{n}$ and $M^{*}$ denotes $\prod M_{n}$.  The natural maps making up the commutative diagram have been labelled in certain cases for convenience. The image of $(w_{n})$ under the composite $\beta\alpha$ is the desired non-zero cohomology class. To check it is non-zero, observe that the image of $(w_{n})$ under the composite $\delta\gamma\beta\alpha$ is the same as its image under the composite $\omega\psi\chi\phi$ and the latter is seen to be non-zero by using the properties (iii) and (iv).}
\label{fig1}
\end{figure}

This completes the proof of III.8.
\end{proof}

\begin{remarks} We used Roseblade's theorem to obtain the splittings that were crucial
in the preceding proof. If our original sequence $N\mono G\epi Q$
was itself split, this appeal to Roseblade's theorem would not have been necessary. Therefore another special case where we can obtain the same conclusion is

\begin{theorem}\label{withD}
Suppose that the finitely generated residually finite group $G$ is the split extension of an infinite abelian torsion group by a constructible soluble group. Then the rational cohomological dimension of $G$ is $h(G)+1$.
\end{theorem}
\end{remarks}

The argument is exactly the same, using as before the two sequences 
 $(T_{n})$ and $(F_{n})$ inside the torsion kernel. Now we use the fact that constructible 
 groups satisfy duality. In this generality there may be a (large) dualising module $D$ and no subgroup of finite index satisfying Poincar\'e duality but this simply has to be carried through the calculation as shown in the diagram.

\begin{figure}[htbp]\[
\xymatrix
{
&M\otimes D\ar@{>>}[r]\ar@{>->}[d]
&M\otimes_QD\ar[r]\ar@{>->}[d]
&\bigoplus(M_n\otimes_QD)\ar@{>->}[d]\\
&M^*\otimes D\ar@{>>}[r]\ar@{>>}[d]
&M^*\otimes_QD\ar[r]\ar@{>>}[d]
&\prod (M_n\otimes_QD)\ar@{>>}[d]\\
\prod M_n^{F_n}\otimes D\ar[ur]\ar[r]\ar@{.>}[dr]
&\dfrac{M^*}{M}\otimes D\ar@{>>}[r]
&\dfrac{M^*}{M}\otimes_QD\ar[r]
&\dfrac{\prod (M_n\otimes_QD)}{\bigoplus (M_n\otimes_QD)}\\
&\left(\dfrac{M^*}{M}\right)^T\otimes D\ar@{>>}[r]\ar@{>->}[u]
&\left(\dfrac{M^*}{M}\right)^T\otimes_QD\ar[u]\\
}
\]
\caption{The diagram is used in the proof of Theorem \ref{withD} in just the same way that the diagram of Figure 1 was employed above.}
\label{fig2}
\end{figure}

\subsection{Uncountable groups}
Consideration of the uncountable case of Conjecture III.1 leads to the following.

\begin{conjecture}\label{c:want} Let $T\mono G\epi Q$ be a short exact sequence of 
countable
groups, where $T$ is infinite, locally finite, and abelian, and where $Q$
is virtually polycyclic with Hirsch length $h(Q)=d$. Then
$H^i(G, F)= 0$ for all free $\Q G$-modules $F$ and all $i\le d$.
\end{conjecture}

\begin{remarks}
(1) This statement is true in the case $d=1$. 
Indeed, if $G$ is finitely generated, then $H^1(G,F)=0$ if and only if $G$
has one end, which $G$ does, since it is not virtually cyclic and does not
contain a non-abelian free group. And in the general case we 
write $G=\bigcup G_n$ with the $G_n$ finitely generated, 
and consider the following short exact sequence of functors, which is in effect the special case of the standard spectral sequence involving the derived functors of $\lim$; details may be found in Jensens' text \cite{Jensen}, see especially Th\'eor\`eme 4.2:
$$
0\to \lim^1 H^{i-1}(G_n, - ) \to H^i(G, -)\to \lim H^i(G_n, -)
\to 0.
$$

\medskip
(2) We are unable to resolve the conjecture in the case $d=2$. In that case,
consideration of the spectral sequence for the group extension $T\mono G\epi G/T$
shows that $H^2(G,F)\cong H^1(G/T, H^1(T,F))$. Alternatively, one might attempt
to proceed by induction on $h(G/Q)$: pass to a subgroup
of finite index and write $G=G_0\rtimes Z_1$ with $G_0=T\rtimes Z_2$
where $Z_1$ and $Z_2$ are infinite cyclic. The spectral sequence for $G_0\rtimes Z_1$
gives $H^2(G,F)\cong H^0(Z_1, H^2(G_0,F))$, and the one for $T\rtimes Z_2$
gives $$H^0(Z_1, H^2(G_0,F))\cong H^1(Z_2, H^1(T,F)).$$ Thus, with either approach, one
is left to understand modules of the form $H^1(T,F)$, with $F$ a free $\Q G$-module.
\end{remarks}

\begin{theorem}\label{t3}
Let $G$ be an uncountable group that is nilpotent-by-polycyclic-by-finite group. Then $G$ has finite rational cohomological dimension if and only if $\hdq(G)<\infty$ and $|G|<\aleph_\omega$. 

If $\hdq(G)<\infty$ and $|G|=\aleph_n$, where $n\in\N$, then $\cdq(G)=\hdq(G)+n+1$,
provided that Conjecture \ref{c:want} is true.
\end{theorem}

The bound $\cdq(G)\le \hdq(G)+n+1$ is covered by the discussion following Theorem II.3. 
The proof of the corresponding lower bound on $\cdq(G)$
will be established by induction on $n$, with Conjecture \ref{c:want} providing the
base case of a subsidiary induction that establishes:

\begin{addendum}\label{a:conj}  
Let $1\to T\to G\to Q\to 1$ be a short exact sequence of 
groups where $G$ has cardinality $\aleph_n$, while
$T$ is infinite, locally finite and abelian, and $Q$
is virtually polycyclic with Hirsch length $h(Q)=h$. Then
$H^i(G, F)= 0$ for all free modules $F$ and all $i\neq h+n+1$. 
\end{addendum}

The strategy of the proof
is motivated by an argument of Derek Holt \cite{holt1981}.   
\medskip

We begin the main argument.
Suppose that $\hdq(G)<\infty$ and $|G|=\aleph_n$. Passing to a 
subgroup of finite index we may assume there is a nilpotent
subgroup $N\normal G$ such that $Q=G/N$ is polycyclic. Note that ${\overline {N}} = N/[N,N]$ has
cardinality at most $|N|=\aleph_n$. Since $h(G)$ is finite, $h({\overline {N}})$ is finite,
so the torsion-free rank of ${\overline {N}}$ is finite. Thus we may choose a countable
subgroup ${\overline {M}}<{\overline {N}}$, invariant under the conjugation action of $G$ such that
${\overline {N}}/{\overline {M}}$ is a torsion group. Let $M$ be the preimage of ${\overline {M}}$ and note
that $N/M = {\overline {N}}/{\overline {M}}$ has cardinality $\aleph_n$. As in the countable case, we
may replace $G$ by $G/M$. In other words, there is no loss
of generality in assuming that $N$ is a locally finite abelian group, and
we do so henceforth. 

Choose a finitely generated $H$ in $G$ such that $HN=G$ (such exists because
$G/N$ is finitely generated). Now $H\cap N$ is {\em normal} in $G$: it is normal
in $N$ because $N$ is abelian, and it is normal in $H$ because $N$ is normal in $G$.
Since $H\cap N$ is countable, quotienting $G$ by $H\cap N$ does not change its
cardinality. Thus, without loss of generality, we may assume that $H\cap N =1$,
which implies that $H$ is polycyclic.

Regard $N$ as a $\Z H$-module in the usual way.
We write $N=\bigcup_{\lambda\in\Lambda} N_\lambda$ as the union of a continuous chain of $\Z H$-submodules of $N$ of cardinality $\aleph_{n-1}$.  Then $G$ is the union of the continuous chain of subgroups $G_{\lambda}=N_{\lambda}\rtimes H$. Here we are supposing that the $\lambda$ run through the ordinals belonging to $\aleph_{n}$, that $N_{\lambda}\subseteq N_{\mu}$ whenever $\lambda<\mu$ and that the chain is continuous in the sense that for each limit ordinal $\lambda\in\aleph_{n}$, there is an equality $N_{\lambda}=\bigcup_{\alpha<\lambda}N_{\alpha}$.
 These same remarks about ordering and continuity also apply to the chain of $G_{\lambda}$.

Note that $h(G_\lambda)=h(G)=h(H)=\hd_\Q(G_\lambda)=\hd_\Q(G)$.

\begin{lemma}\label{Holt's condition star}
The chain of subgroups $G_{\lambda}$, ($\lambda\in\aleph_{n}$), can be chosen so that for each $\lambda$ there is a subgroup $G_{\lambda}'$ of $G_{\lambda+1}$ that contains $G_{\lambda}$ as a proper subgroup of finite index.
\end{lemma}
\begin{proof}
Let $H$, $N_{\lambda}$ be chosen as explained above. The initial choice $G_{\lambda}:=N_{\lambda}H$ may not satisfy the conclusion of our lemma but we start with this and make adjustments.
Clearly we may assume that $N_{\lambda}<N_{\lambda+1}$ for each $\lambda\in\aleph_{n}$.
For each $\lambda$, choose $x_{\lambda}\in N_{\lambda+1}\smallsetminus N_{\lambda}$ and let $M_{\lambda}$ be a submodule of $N_{\lambda+1}$ containing $N_{\lambda}$, and chosen using Zorn's Lemma to be maximal subject to not containing $x_{\lambda}$. Then $N_{\lambda+1}/M_{\lambda}$ has a simple socle $L_{\lambda}/M_{\lambda}$ and this is finite by Roseblade's Main Theorem \cite {roseblade}. Now we may replace $N_{\lambda}$ by $M_{\lambda}$ and $G_{\lambda}=N_{\lambda}H$ by $M_{\lambda}H$. The result follows because each $M_{\lambda}H$ is properly contained in the finite index overgroup $L_{\lambda}H$ and this in turn is contained in $N_{\lambda+1}H\subset M_{\lambda+1}H$.
\end{proof}

We consider the cohomology of $G$ with coefficients in an arbitrary free $\Q G$-module. For a $\Q$-vector space $V$ we write $VG$ for the free module $V\otimes\Q G$ and similarly for any subgroup $S$ of $G$ we write $VS$ for the free $\Q S$-submodule $V\otimes\Q S$. 

We work with the standard bar resolution for cohomology. Thus cohomology classes $[\theta]\in H^{n}(G,M)$ are represented by $n$-cocycles $$\theta:\underbrace{G\times\dots\times G}_n\to M.$$ Details may be found in any standard text on group cohomology, see for example \cite{gruenberg}.

\begin{lemma}\label{l:d} Let $d$ be a natural number.
Suppose that the cohomology groups $H^{d-1}(G_{\lambda},VG)$ vanish for  all $\lambda$. If $\theta:\underbrace{G\times\dots\times G}_{d}\to VG$ is a $d$-cocycle whose restriction to $\underbrace{G_{\lambda}\times\dots\times G_{\lambda}}_{d}$ is a coboundary for all $\lambda$, then $\theta$ is a coboundary. In other words, every non-zero element of $H^{d}(G,VG)$ has non-zero restriction to $H^{d}(G_{\lambda},VG)$ for some $\lambda$.
\end{lemma}
\begin{proof}
This is a well known argument but we include the details for the reader's convenience. For each $\lambda$ we may choose a cochain 
$$\phi_{\lambda}:\underbrace{G_{\lambda}\times\dots\times G_{\lambda}}_{d-1}\to VG$$ such that $\delta\phi_{\lambda}$ agrees with $\theta$ on $\underbrace{G_{\lambda}\times\dots\times G_{\lambda}}_{d}$. We wish to make the choices of $\phi_{\lambda}$ in a compatible way so that if $\lambda<\lambda'$ then $\phi_{\lambda}$ and $\phi_{\lambda'}$ agree on $\underbrace{G_{\lambda}\times\dots\times G_{\lambda}}_{d-1}$. If this can be done then clearly the $\phi_{\lambda}$ uniquely determine a cochain $\phi$ and it has the property $\delta\phi=\theta$. 

We make the choices of $\phi_{\lambda}$ inductively. Suppose that $\phi_{\lambda}$ is to be chosen and that $\phi_{\alpha}$ have already been chosen in a compatible way for each $\alpha<\lambda$. If $\lambda$ is a limit ordinal then clearly there is a unique choice for $\phi_{\lambda}$ specified by the fact that its domain $G_{\lambda}\times\dots\times G_{\lambda}$ is the union of the domains $G_{\alpha}\times\dots\times G_{\alpha}$ of the already defined $\phi_{\alpha}$, ($\alpha<\lambda$). If $\lambda$ is a successor ordinal then begin by choosing any cochain $\psi:G_{\lambda}\times\dots\times G_{\lambda}\to VG$ such that $\delta\psi$ agrees with $\theta$ on $G_{\lambda}\times\dots\times G_{\lambda}$. We have that $\delta\phi_{\lambda-1}$ and $\delta\psi$ agree with $\theta$, and therefore with each other, on
$G_{\lambda-1}\times\dots\times G_{\lambda-1}$, and hence that the restriction of $\psi-\phi_{\lambda-1}$ to this domain is a $(d-1)$-cocycle: it satisfies $\delta(\psi-\phi_{\lambda-1})=0$. Now we invoke the hypothesis that $H^{d-1}(G_{\lambda-1},VG)$ vanishes. Therefore there is a $(d-2)$-cochain $\xi:\underbrace{G_{\lambda-1}\times\dots\times G_{\lambda-1}}_{d-2}$ such that $\delta\xi=\psi-\phi_{\lambda-1}$. Now choose any extension $\zeta$ of $\xi$ to $\underbrace{G_{\lambda}\times\dots\times G_{\lambda}}_{d-2}$ and set $\phi_{\lambda}$ equal to $\psi-\delta\zeta$. On $\underbrace{G_{\lambda}\times\dots\times G_{\lambda}}_{d}$ we have $\delta\phi_{\lambda}=\delta(\psi-\delta\zeta)=\delta\psi=\theta$ and also $\phi_{\lambda}$ extends $\psi-\delta\xi=\phi_{\lambda-1}$ on $\underbrace{G_{\lambda-1}\times\dots\times G_{\lambda-1}}_{d-1}$.
\end{proof}

\begin{lemma}\label{l:lambda-mu}
Let $[\theta]$ be a cohomology class in $H^{d}(G,VG)$ represented by a cocycle $\theta$. Then for each $\lambda$ there exists $\mu\ge\lambda$ such that the restriction of $[\theta]$ to $H^{d}(G_{\mu},VG)$ lies in the summand $H^{d}(G_{\mu},VG_{\mu})$.
\end{lemma}
\begin{proof} This argument may be found in Holt's paper \cite{holt1981}.
Fix $\lambda$.
Define a sequence $\lambda_{j}$ of ordinals in $\aleph_{n}$, indexed by natural numbers $j$ and starting with $\lambda_{1}=\lambda$, so that the restriction of $\theta$ to $G_{\lambda_{j}}\times\dots\times G_{\lambda_{j}}$ has image in the summand $VG_{\lambda_{j+1}}$. Then let $\mu$ be the limit (i.e. union) of the $\lambda_{j}$.
\end{proof}

\begin{lemma}\label{l:Holt*} Suppose that $K<L \le P\le G$ are groups 
 where $K$ has finite index in $L$.
Then, for all $[\theta]\in H^d(P,VG)$, if the restriction of $[\theta]$ to $H^{d}(K,VG)$ lies in the summand $H^{d}(K,VK)$ then this restriction is zero.
\end{lemma}
\begin{proof}
This is an easy variation on Holt's (\cite{holt1981}, Lemma 3).
\end{proof}

We shall now complete the promised deduction of Theorem \ref{t3} and Addendum \ref{a:conj} from Conjecture \ref{c:want}.

\begin{proposition}
Conjecture \ref{c:want} implies that if $G$ is a nilpotent-by-polycyclic-by-finite group of cardinality $\aleph_n$ (where $0<n<\omega$) and Hirsch length $h<\infty$ and $V$ is any vector space over $\Q$ then $H^{i}(G,VG)=0$ for all $i\ne n+h+1$.
\end{proposition}
\begin{proof} 
In the paragraphs following the statement of Addendum \ref{a:conj}
we discussed how to reduce  
to the case of a (torsion abelian)-by-polycyclic group. 
Now, provided we insist that the torsion-abelian subgroup is infinite, the statement of the Proposition
makes sense for groups of this form when $n=0$, in which case it is a direct translation from Conjecture \ref{c:want}.
This observation  forms the base case for a proof by induction. 

Assume now that $n\ge 1$ and that $G = \bigcup_{\lambda\in\Lambda}G_\lambda$, as arranged above.
By induction, $H^{d-1}(G_\lambda, VG)=0$ for all $d\le h+n$, so by Lemma \ref{l:d} the induction will
be complete if the restriction to every $G_\lambda$ of each cohomology class $[\theta]\in H^{h+n}(G,VG)$ is zero.
It is enough to show, for each $\lambda$, that there exists $\mu>\lambda$
such that the restriction to $H^{h+n}(G_\mu, VG)$ is trivial, and in the light of Lemma \ref{l:lambda-mu} 
we may assume that this last restriction lies in the summand $H^{h+n}(G_\mu, VG_\mu)$. Finally,
Lemma \ref{Holt's condition star} allows us to employ Lemma \ref{l:Holt*} with $K=G_\mu$ and
$L=G_\mu'$ and $P=G$, and thus we conclude that the restriction of $[\theta]$ to $H^{h+n}(G_\mu, VG)$ 
is zero, as required. 
\end{proof}

\bibliography{brid}
\bibliographystyle{abbrv}

\end{document}